\documentclass{amsart}   %%%%%%%Please do not change
\usepackage{graphicx, amsmath, amssymb, amsthm, amscd}
\usepackage{enumerate}
\usepackage{color}

\usepackage[all]{xy}
\usepackage{afterpage}
\usepackage{multirow}

\newtheorem{theorem}{Theorem}[section]
\newtheorem{corollary}[theorem]{Corollary}
\newtheorem{lemma}[theorem]{Lemma}

\newtheorem{question}{Question}

\theoremstyle{definition}

\newtheorem{claim}{Claim}[theorem]

\begin{document}

\title[Continuous curves of nonmetric pseudo-arcs]{Continuous curves of nonmetric pseudo-arcs and semi-conjugacies to interval maps}

\author{Jan P. Boro\'nski}
\address[J. P. Boro\'nski]{Faculty of Applied Mathematics,
AGH University of Science and Technology,
al. Mickiewicza 30,
30-059 Krak\'ow,
Poland -- and -- National Supercomputing Centre IT4Innovations, Division of the University of Ostrava,
Institute for Research and Applications of Fuzzy Modeling,
30. dubna 22, 70103 Ostrava,
Czech Republic}
\email{jan.boronski@osu.cz}
\author{Michel Smith}
\address[Michel Smith]{Department of Mathematics and Statistics, Auburn University, AL 36849, USA}
\email{smith01@auburn.edu}

\begin{abstract}
In 1985 M. Smith constructed a nonmetric pseudo-arc; i.e. a Hausdorff homogeneous, hereditary equivalent and hereditary indecomposable continuum. Taking advantage of a decomposition theorem of W. Lewis, he obtained it as a long inverse limit of metric pseudo-arcs with monotone bonding maps. Extending his approach, and the results of Lewis on lifting homeomorphisms, we construct a nonmetric pseudo-circle, and new examples of homogeneous 1-dimensional continua; e.g. a circle and solenoids of nonmetric pseudo-arcs. Among many corollaries we also obtain an analogue of another theorem of Lewis from 1984: any interval map is semi-conjugate to a homeomorphism of the nonmetric pseudo-arc.

\end{abstract}
\maketitle

\section{Introduction}
The present paper is motivated by a well-known open problem in topology.
\vspace{0.25cm}

\noindent
\textbf{Problem:} \textit{Classify homogeneous compacta.}
\vspace{0.25cm}

\noindent
Mislove and Rogers showed in \cite{MisloveRogers} that each homogeneous compact metric space is the Cartesian product of a homogeneous continuum $Y$
and the set $C$, where $C$ is the Cantor set, or a finite space. Recall that a compact space $X$ is \textit{homogeneous} if for any two points $x$ and
$y$ in $X$ there exists a homeomorphism $h: X \to X$ such that $h(x) = y$. Among the well-known
homogeneous compact spaces there are objects such as the circle, Cantor set, surfaces of the sphere and torus, and 1-dimensional Vietoris solenoids. Finite spaces are homogenous in a trivial way. Despite the fact that some spaces are not homogeneous, they are very close to being so and in this case it is useful to look at their \textit{degree of homogeneity}. We say that the degree of homogeneity of the space $X$ is $1/n$, where $n$ is a cardinal number, when the group of homeomorphisms of $X$ has exactly $n$ orbits; i.e. $X$ has exactly $n$ topologically distinct types of points. A classical nontrivial example of a $1/2$-homogeneous continuum is the Sierpi\'nski carpet \cite{Krasinkiewicz}. There is another group of compact spaces that are very close to being homogeneous in the sense that their homeomorphism groups act minimally on them. Namely, we say that $X$ is \textit{nearly homogeneous} when for any pair of points $x$ and $y$ in $X$ and any open neighborhood $U$ of $y$ there exists a homeomorphism $h:X \to X$ such that $h(x)$ belongs to $U$. There are examples of spaces that are nearly homogeneous, although their degree of homogeneity is $1/n$ for no $n\in\mathbb{N}$. One such example is a compactum that belongs to a class of \textit{hereditarily indecomposable} continua, particularly important from the view point of this study. We say that $Y$ is an \textit{indecomposable} continuum if it cannot be represented as the sum of two proper subcontinua. Continuum $Y$ is \textit{hereditarily indecomposable} if every subcontinuum is indecomposable. Two particularly prominent examples of hereditarily indecomposable continua are a \textit{pseudo-arc }and \textit{pseudo-circle}, sometimes referred to as bad fractals due to their very complex structure and certain degree of self-similarity. Pseudo-arc is a space first constructed by B. Knaster in the 20s of the XX century \cite{Knaster}. It is a one-dimensional planar continuum, that does not separate the plane, and is homeomorphic to any
of its subcontinua \cite{Mo}. In the 40s R.H. Bing showed that the pseudo-arc is homogeneous \cite{Bi2} and later, with
Jones, he constructed a homogeneous (decomposable) circle of pseudo-arcs \cite{Bi3}. Bing also constructed
the pseudo-circle \cite{Bi}, the unique hereditarily indecomposable planar cofrontier. The pseudo-circle is
nearly homogeneous \cite{KeRo}, though neither homogeneous \cite{Fe}, \cite{Ro2} nor $1/n$-homogeneous for any $n>1$. Lewis showed that for any 1-dimensional continuum $Y$ there exists a continuum $Y'$ that admits a continuous and monotone decomposition into pseudoarcs, and if $Y$ is homogeneous then so is $Y'$ \cite{Le}. Most recently in \cite{HoOv} Hoehn and Oversteegen classified all planar homogeneous continua as: the circle, pseudo-arc, and circle of pseudo-arcs. It is noteworthy that the pseudo-arc appeared quite recently also in studies concerning other areas of mathematics, such as functional analysis and logic. In 2005, Rambla \cite{Rambla} and Kawamura \cite{Kawamura} showed independently, that the pseudo-arc minus a point provides a counterexample to the Conjecture of Wood, formulated in the early 80’s, about the existence of certain Banach spaces with transitive group of isometries. About the same time Solecki and Irwin \cite{Solecki} obtained the pseudo-arc as the quotient space of a projective Fr\"iss\'e limit of reflexive linear graphs. Expanding their tools, in 2014 Kwiatkowska showed that the pseudo-arc has a dense conjugacy class of homeomorphisms \cite{Kwiatkowska}, reproving the result of Oppenheim from 2008 \cite{Op}. These findings further confirm the importance of research on hereditary indecomposable spaces and their significance beyond a narrow area of mathematics. Moreover Smith's nonmetric pseudo-arc has recently provided another counterexample to Wood's conjecture \cite{BoronskiSmith}, which highlights the importance of including the nonmetric hereditarily indecomposable continua in future studies. In the present paper, building on the approach of the second author, and heavily relying on Lewis' results for the metric pseudo-arc, we extend Lewis' results to prove the following.
\vspace{0.25cm}

\noindent
\textbf{Theorem A.} \textit{For every one-dimensional metric continuum $M$, there exists a continuum $M_{\omega_1}$ such that $M_{\omega_1}$ has a continuous decomposition into nonmetric pseudo-arcs, and the decomposition space is homeomorphic to $M$. Additionally:
\begin{enumerate}
	\item If $M$ is homogeneous, then so is $M_{\omega_1}$.
	\item Every homeomorphism of $M$ lifts to a semi-conjugate homeomorphism of $M_{\omega_1}$.
	\item $M_{\omega_1}$ admits a Whitney map onto a long arc.
\end{enumerate}
}
\vspace{0.25cm}

\noindent
We also give an analogue of another result of Lewis. In \cite{WL} he showed that any self-map of an arc-like continuum is semi-conjugate to a pseudo-arc homeomorphism. In the present paper we exploit his result to connect one-dimensional dynamics to the dynamics on Smith's nonmetric pseudo-arc. 
\vspace{0.25cm}

\noindent
\textbf{Theorem B.} \textit{Every interval map is semi-conjugate to a homeomorphism of Smith's nonmetric pseudo-arc.
}
\vspace{0.25cm}

The paper is organized as follows. In Section 2 we recall some preliminary facts, including the results of Lewis and properties of Smith's nonmetric pseudo-arc. In Section 3 we describe the construction of continuous curves of nonmetric pseudo-arcs and discuss some curious applications of the method. One of them is the aforementioned result on the semi-conjugacy of interval maps with homeomorphisms of the nonmetric pseudo-arc. In Section 4 we show that our long inverse limit approach cannot be continued beyond $\omega_1$, as in a sense it stabilizes at this step. In Section 5 we observe that any continuous curve of nonmetric pseudo-arcs from Section 3 admits a generalized Whitney map onto the long arc. We conclude our paper in Section 6 with some questions on Smith's nonmetric pseudo-arc. 

\section{Preliminary results}
Recall that given a graph $\mathcal{G}$,  a continuum $X$ is called $\mathcal{G}$-like if for each open cover $\mathcal{U}$ of $X$ there is a finite open refinement $\{U_1,\ldots,U_t\}$ whose nerve is $\mathcal{G}$. If $\mathcal{G}=[0,1]$ then $X$ is called arc-like or chainable. Below we recall the results of Lewis, on which we shall heavily rely.
\begin{theorem}[Lewis \cite{Le}] \label{WL1} Suppose that $M$ is a one-dimensional metric continuum.  Then there exists a one-dimensional continuum $\hat M$ and a continuous decomposition $G$ of $\hat M$ into pseudo-arcs so that the decomposition space $\hat M/G$ is homeomorphic to $M$.  Furthermore if $\pi: \hat M \rightarrow \hat M/G$ is the mapping so that $\pi(x)$ is the unique element of $G$ containing $x$ and $h: \hat M/G \rightarrow \hat M/G$ is a homeomorphism then there exists a homeomorphism $\hat h: \hat M \rightarrow \hat M$ so that $\pi \circ \hat h = h \circ \pi$.
\end{theorem}

Thus we have an open and closed map $\pi$ so that the following diagram commutes:
\begin{center}
\begin{displaymath}
\xymatrix{
\hat M \ar@{<-}[r]^{\hat h }\ar@{->}[d]^{\pi}& \hat M \ar@{->}[d]^{\pi}\\
\hat M / G = M \ar@{<-}[r]^{h}& M = \hat M / G
}
\end{displaymath}
\end{center}

\begin{theorem}[Lewis \cite{Le}]\label{WL2} Under the hypothesis of Theorem \ref{WL1},  if for some element $g \in G$ we have $x, y \in g$ then there exists a homeomorphism $\hat h$ so that $\hat h(x) = y$ and $\pi \circ \hat h = \pi$.
\end{theorem}

\begin{theorem}[Lewis \cite{Le}]\label{WL3}  Suppose that $X$ is a pseudo-arc and $G$ is a continuous collection of pseudo-arcs filling up $X$ so that for each $x \in X$, $\pi(x)$ is the element of $G$ containing $x$ and $Y = X/G$.  Then $Y$ is a pseudo-arc, and if $h: Y \rightarrow Y$ is a homeomorphism then there exists a homeomorphism $\hat h: X \rightarrow X$ so that $\pi \circ \hat h = h \circ \pi$.
\end{theorem}
\begin{center}
\begin{displaymath}
\xymatrix{
X \ar@{<-}[r]^{\hat h }\ar@{->}[d]^{\pi}& X \ar@{->}[d]^{\pi}\\
Y \ar@{<-}[r]^{h}& Y
}
\end{displaymath}
\end{center}
\begin{theorem}[Lewis \cite{Le}]\label{semiconjugacy}
Suppose $X$ is a metric chainable continuum and $f: X\to X$ is a map. Then there exists a pseudo-arc homeomorphism $h:P\to P$ and a continuous surjection $p:P\to X$ such that $f \circ p=p\circ h$. In particular, $p$ gives a semi-conjugacy between $f$ and $h$.
\end{theorem}
The following result will also be essential to our proofs.
\begin{theorem}\cite{En} Suppose that $X = \varprojlim\{X_\alpha, f_\alpha^\beta \}_{\alpha < \beta < \omega_1}$ and $Y = \varprojlim\{Y_\alpha, G_\alpha^\beta \}_{\alpha < \beta < \omega_1}$ are inverse limits and there is an index $\gamma < \omega_1$ and a collection of mappings $\varphi_\delta$ for $\delta > \gamma$ the following diagram commutes:
\begin{center}
\begin{displaymath}
\xymatrix{
X_\gamma \ar@{<-}[r]^{ f_\gamma^\delta }\ar@{->}[d]^{\varphi_\gamma}& X_\delta \ar@{->}[d]^{\varphi_\delta}\\
Y_\gamma \ar@{<-}[r]^{g_\gamma^\delta}& Y_\delta
}
\end{displaymath}
\end{center}
Then there is an induced mapping $\varphi: X \rightarrow Y$ so that the following commute:
\begin{displaymath}
\xymatrix{
X_\gamma \ar@{<-}[r]^{ \pi_\gamma }\ar@{->}[d]^{\varphi_\gamma}& X \ar@{->}[d]^{\varphi}\\
Y_\gamma \ar@{<-}[r]^{\pi_\gamma }& Y
}
\end{displaymath}
\end{theorem}
In \cite{MS} the second author constructed a nonmetric Hausdorff hereditarily indecomposable arc-like continuum, as an $\omega_1$-long inverse limit of metric pseudo-arc with monotone, open and closed bonding maps, since then referred to as Smith's nonmetric pseudo-arc. The following theorem lists some of the known properties of that space.
\begin{theorem}
Let $P_{\omega_1}$ be Smith's nonmetric pseudo-arc. Then 
\begin{enumerate}
	\item $P_{\omega_1}$ is homogeneous and hereditarily equivalent \cite{MS},
	\item $P_{\omega_1}$ is separable \cite{BoronskiSmith},
	\item $P_{\omega_1}$ is not first countable at any point \cite{BoronskiSmith},
	\item $P_{\omega_1}$ contains a convergent sequence \cite{BoronskiSmith}. 
\end{enumerate}
\end{theorem}
In the next section we are going to generalize Smith's construction to other metric continua. 
\section{Continuous curves of nonmetric pseudo-arcs}
\begin{theorem}
For every one-dimensional metric continuum $M$, there exists a continuum $M_{\omega_1}$ such that $M_{\omega_1}$ has a continuous decomposition into nonmetric pseudo-arcs, and the decomposition space is homeomorphic to $M$.
\end{theorem}
\begin{proof}
Let $M_0=M$ be a one-dimensional continuum. By Lewis' decomposition theorem there is a continuum $M_1$, a collection $G_1$ of mutually disjoint pseudo-arcs in $M_1$, and an open and closed map $p_0:M_1\to M_0$ such that $x\in M$ if and only if the point-inverse $p_0^{-1}(x)\in G_1$; furthermore if $P_0: M_1/G_1 \to M_0$ is defined by $P_0(g)$ is the unique element of $M_0$ so that $p_0^{-1}(x) = g$ then $P_0$ is a homeomorphism.

Suppose that $n>1$ is an integer and $M_n, G_n, p_{n-1}, P_{n-1}$,   have been constructed.  Then, by Lewis' theorem, there is a continuum $M_{n+1}$, a collection $G_{n+1}$ of mutually disjoint pseudo-arcs in $M_{n+1}$, and an open and closed map $p_{n}:M_{n+1} \to M_n$ such that $x_n \in M_n$ if and only if the point-inverse $p_{n}^{-1}(x_n)\in G_{n+1}$ and if $P_n: M_{n+1}/G_{n+1} \to M_n$ is defined by $P_n(g)$ is the unique element of $M_n$ so that $p_n^{-1}(x) = g$ then $P_n$ is a homeomorphism.  For each pair of integers $m<n$ define $p_m^n = p_m \circ p_{m+1} \circ \ldots \circ p_{n-1}$; let $M_\omega = \varprojlim\{M_n, p_m^n\}_{m<n < \omega}$ and define $p_n^\omega$ to be the projection of $M_\omega$ onto the $n^{\mbox{th}}$ coordinate space.  Observe that if $m<n$ then $p_m^\omega = p_m^n \circ p_n^\omega$ and that as a projection the map $p_n^\omega$ is closed and open.

We continue by transfinite induction to construct $M_\alpha$ for every $\alpha<\omega_1$ as in \cite{MS}.  Suppose that $\alpha$ is an ordinal and that $M_\beta, G_\beta, p_\gamma^\beta$,   have been constructed for all $\gamma < \beta < \alpha$.

\

Case 1. $\alpha$ is a limit ordinal.  Then define:
\begin{eqnarray*}
M_\alpha & = & \varprojlim\{M_\gamma, p_\gamma^\beta\}_{\gamma<\beta<\alpha}
\end{eqnarray*}
and observe that
\begin{eqnarray*}
p_\gamma^\alpha & = & p_\gamma^\beta \circ p_\beta^\alpha
\end{eqnarray*}
where $p_\beta^\alpha: M_\alpha \to M_\beta$ is the projection map.

\

Case 2. $\alpha$ is not a limit ordinal.  Since $\alpha$ is a countable ordinal, $M_{\alpha -1}$ is a one-dimensional metric continuum.  Then by Lewis' decomposition theorem there is a continuum $M_{\alpha}$, a collection $G_{\alpha}$ of mutually disjoint pseudo-arcs in $M_{\alpha}$, and an open and closed map $p_{\alpha-1}^{\alpha} :M_{\alpha} \to M_{\alpha-1}$ such that $x\in M_{\alpha-1}$ if and only if the point-inverse ${p_{\alpha-1}^{\alpha}}^{-1}(x)\in G_{\alpha}$; furthermore if $P_{\alpha-1}: M_{\alpha}/G_{\alpha} \to M_{\alpha-1}$ is defined by $P_{\alpha -1}(g)$ is the unique element of $G_{\alpha}$ so that ${p_{\alpha-1}^\alpha}^{-1}(x) = g$ then $P_{\alpha-1}$ is a homeomorphism.
We define
\begin{eqnarray*}
p_\gamma^\alpha & = & p_\gamma^{\alpha-1} \circ p_{\alpha-1}^\alpha.
\end{eqnarray*}
Let $M_{\omega_1}=\varprojlim\{M_\alpha,p_\beta^\alpha\}_{\beta<\alpha<\omega_1}$ and let $p_\alpha^{\omega_1}$ be the natural projection onto the $\alpha$th coordinate space $M_\alpha$.  We now show that $M_{\omega_1}$ is a continuum with the desired properties.

\

\begin{claim}\label{open}
$p_0^{\omega_1}$ is open.
\end{claim}
\begin{proof}(of Claim \ref{open})
Since $p_0^{\omega_1}$ is a projection onto the coordinate space $M_0$ of the inverse limit, it is open.
\end{proof}
\begin{claim}\label{closed}
$p_0^{\omega_1}$ is closed.
\end{claim}
\begin{proof}(of Claim \ref{closed})
This follows from the fact that $M_{\omega_1}$ is a closed subset of the Cartesian product $\prod_{\gamma< \omega_1}{M_\gamma}$ and $p_0^{\omega_1}$ coincides with the projection onto the $0$th coordinate, with $\prod_{\gamma< \omega_1}{M_\gamma}$ compact.
\end{proof}
\begin{claim}\label{point}
For every $x\in M_0$ and $0 < \alpha < \omega_1$ the point-inverse $p_0^\alpha(x)$ is a pseudo-arc.
\end{claim}
\begin{proof}(of Claim \ref{point})
Fix $x\in M_0$. Then ${(p_0^1)}^{-1}(x)$ is a pseudo-arc. Suppose that $\alpha<\omega_1$ and for $\beta<\alpha$,  $p_0^\beta$ is a pseudo arc.  For $\alpha$ a limit ordinal ${(p_0^\alpha)}^{-1}(x)$ will be homeomorphic to $\varprojlim\{ {p_0^\beta}^{-1}(x), p_\gamma^\alpha\}_{0 < \gamma<\beta < \alpha}$.  But this is the countable inverse limit of hereditarily indecomposable chainable metric continua and so it must be a metric chainable hereditarily indecomposable continuum as well, and so it is the pseudo-arc by its uniqueness. For $\alpha$ a successor ordinal, ${(p_0^\alpha)}^{-1}(x)=(p_{(\alpha-1)}^\alpha)^{-1}\big( \big(p_0^{(\alpha-1)}\big)^{-1}(x)\big)$ is a continuous decomposition of the pseudo-arc $\big(p_0^{(\alpha-1)}\big)^{-1}(x)$ into
pseudo-arcs $G_\alpha = \{ \big(p_{(\alpha-1)}^\alpha \big)^{-1}(t) | t \in \big( p_0^{(\alpha-1)}\big)^{-1}(x) \}$ and so is a pseudo-arc by Theorem \ref{WL3}.
\end{proof}

Consequently $M_{\omega_1}=\lim_\leftarrow\{M_\alpha,p_\alpha^\beta:\alpha<\beta<\omega_1\}$ is a space with a natural open and closed projection map $p_0^{\omega_1}$ onto $M_0$, where for each $x\in M_0$ the point-inverse $\big(p_0^{\omega_1} \big)^{-1}(x)$ is the nonmetric pseudo-arc constructed in \cite{MS}. For each $x \in M_0$ let $g^{\omega_1}_x = \{ p \in M_{\omega_1} | p_0 = x \}$. Therefore if $G_{\omega_1} = \{g_x^{\omega_1} | x \in M_0 \}$ then $G_{\omega_1}$ is a continuous decomposition of $M_{\omega_1}$ into nonmetric pseudo-arcs whose decomposition space $M_{\omega_1}/G_{\omega_1}$ is $M_0$.
\end{proof}
It follows from Lewis' construction that if $M$ is $\mathcal{G}$-like then his ``$M$ of metric pseudo-arcs'' is $\mathcal{G}$-like as well. Because of the inverse limit construction it follows that the nonmetric $M_{\omega_{1}}$ will be $\mathcal{G}$-like as well. Moreover, since $p_0^{\omega_1}$ has acyclic fibers, by Vietoris--Begle mapping theorem \cite{Sp}, the 1st \v Cech cohomology groups of $M_{\omega_{1}}$ and $M$ will be isomorphic. An interesting consequence is the existence of a nonmetric pseudo-circle.
\begin{corollary}\label{pseudocircle}
There exists a nonmetric pseudo-circle $C_{\omega_{1}}$; i.e. a hereditarily indecomposable, nonchainable circle-like continuum such that its 1st \v Cech cohomology group is $\mathbb{Z}$. In addition, there is an open map $\pi_{\omega_1}:C_{\omega_{1}}\to C$, where $C$ is the pseudo-circle.
\end{corollary}
\begin{theorem}\label{homeolift}
If $h:M\to M$ is a homeomorphism then there exists a lift homeomorphism $\bar h:M_{\omega_1}\to M_{\omega_1}$ such that $h\circ \pi=\pi\circ \bar h$, where $\pi:M_{\omega_1}\to M$ is the open and closed quotient map.
\end{theorem} 
\begin{proof}
We proceed by transfinite induction applying Lewis' result for the metric pseudo-arc to the above inverse limit construction. Let $h :M\to M$ be the given homeomorphism.  Then by Theorem 2.1 there exists a homeomorphism $h_2: M_2 \to M_2$ so that the following diagram commutes:
\begin{center}
\begin{displaymath}
\xymatrix{
M_2 \ar@{<-}[r]^{h_2 }\ar@{->}[d]^{p^2_1}&  M_2 \ar@{->}[d]^{p_1^2}\\
 M \ar@{<-}[r]^{h}& M
}
\end{displaymath}
\end{center}
Suppose that $h_\alpha$ has been constructed for all $\gamma < \alpha$ so that for $\beta < \gamma$ we have $p^\gamma_\beta \circ h_\gamma = h_\beta p^\gamma_\beta$.

Case 1:  $\alpha=\gamma+1$ is a successor ordinal.  Then by Theorem 2.1
there exists a homeomorphism $h_\alpha: M_\alpha \to M_\alpha$ so that the following diagram commutes:
\begin{center}
\begin{displaymath}
\xymatrix{
M_\alpha \ar@{<-}[r]^{h_\alpha }\ar@{->}[d]^{p^\alpha_\gamma}&  M_\alpha \ar@{->}[d]^{p^\alpha_\gamma}\\
 M_\gamma \ar@{<-}[r]^{h_\gamma}& M_\gamma
}
\end{displaymath}
\end{center}
Since for $\beta < \gamma$ we have $p^\gamma_\beta \circ h_\gamma = h_\beta p^\gamma_\beta$; then this together with the above gives us $p^\alpha_\beta \circ h_\alpha = h_\beta p^\alpha_\beta$.

Case 2: $\alpha$ is a limit ordinal.  Then by construction $M_\alpha = \varprojlim\{M_\gamma, p_\gamma^\beta\}_{\gamma<\beta<\alpha}$.  Define $h_\alpha$ coordinate-wise by $h_\alpha(\{x_\gamma\}_{\gamma < \alpha}) = \{h_\gamma(x_\gamma)\}_{\gamma < \alpha}$.  Then by the properties of inverse limits, $h_\alpha$ will be a homeomorphism and for $\gamma < \alpha$ we have $p^\alpha_\gamma \circ h_\alpha = h_\gamma p^\alpha_\gamma$.

Now we define $h_{\omega_1}$ coordinate-wise by $h_{\omega_1}(\{x_\gamma\}_{\gamma < {\omega_1}}) = \{h_\gamma(x_\gamma)\}_{\gamma < {\omega_1}}$.  Which by the inductive formula that gives us $p^\alpha_1 \circ h_\alpha = h_1 p^\alpha_1$ where $h_1 = h$ we will have $p^{\omega_1}_1 \circ h_{\omega_1} = h_{\omega_1} \circ p^{\omega_1}_1$ as required.
\end{proof}
Similarly, using Theorem \ref{WL2}, one proves the following. The proof is left to the reader. 
\begin{theorem}
If $x, y \in M_{\omega_1}$ are such that $p^{\omega_1}_1(x)=p^{\omega_1}_1(y)$ then there exists a homeomorphism $\hat h$ so that $\hat h(x) = y$ and $p^{\omega_1}_1 \circ \hat h = p^{\omega_1}_1$.
\end{theorem}
\begin{corollary}
If $M$ is homogeneous then so is $M_{\omega_1}$.
\end{corollary}
\begin{corollary}
There exists a homogeneous circle of nonmetric pseudo-arcs, and homogeneous solenoids of nonmetric pseudo-arcs.
\end{corollary}
\begin{corollary}
The nonmetric pseudo-circle from Corollary \ref{pseudocircle} is nearly homogeneous. 
\end{corollary}
The following corollary can be proved following the arguments presented here, combined with those in the proof of Lemma 4.1 in \cite{Bo}.
\begin{corollary}
For every $n>0$ there exists a $\frac{1}{n}$-homogeneous solenoidal continuum of nonmetric pseudo-arcs.
\end{corollary}
\begin{proof}
Fix $n>0$ and let $S_n$ be a hereditarily decomposable $\frac{1}{n}$-homogeneous solenoidal continuum constructed in \cite{RiPP}. Then by Lemma 4.1 in \cite{Bo} there exists a $\frac{1}{n}$-homogeneous solenoidal continuum of pseudo-arcs. Now proceed to the long inverse limit. The details are left to the reader as an exercise.
\end{proof}

\begin{theorem}
Suppose $X$ is a metric chainable continuum and $f: X\to X$ is a map. Then there exists a homeomorphism of the nonmetric pseudo-arc $h:P_{\omega_1}\to P_{\omega_1}$ and a continuous surjection $p:P_{\omega_1}\to X$ such that $f \circ p=p\circ h$.
\end{theorem}

\begin{proof}
Let $M$ be the pseudo-arc $P$, then applying Theorem \ref{semiconjugacy} there exists a homeomorphism  $h:M \to M$ and a continuous surjection $q:M \to X$ such that $f \circ q=q\circ h$.  Then from Theorem \ref{homeolift} there is a homeomorphism $h_{\omega_1} :M_{\omega_1}\to M_{\omega_1}$ such that $h_{\omega_1}\circ p_1^{\omega_1}=p_1^{\omega_1} \circ h_{\omega_1}$.  It is straightforward to verify that $h_{\omega_1}$ will be the required map.
\end{proof}
\section{More on decompositions of $M_{\omega_1}$}
In this section we are interested if the procedure described in the previous section could be extended beyond $\omega_1$, as to produce other nonmetric pseudo-arcs and corresponding new decompositions. As we shall show, this is unfortunately not possible by the same approach, as any decomposition space of $M_{\omega_1}$ must be a metric continuum.
\begin{theorem} Let $M_{\omega_1}=\varprojlim\{M_\alpha, f_\alpha^\beta\}_{\alpha<\beta<\omega_1}$ be the non-metric one-dimensional continuum constructed above.  Suppose that $G$ is a continuous collection of nondegenerate continua that fills up $M_{\omega_1}$.  Then $M_{\omega_1}/G$ is a one-dimensional metric continuum.
\end{theorem}
\begin{proof}
Recall: $M_0$ is an arbitrary one dimensional metric continuum.  For $\alpha > 0$, $\alpha$ a successor ordinal,
$G_\alpha = \{ f_{\alpha-1}^{-1}(p) | \ p \in X_{\alpha-1} \}$ is the continuous decomposition of $M_{\alpha-1}$ used to construct $M_\alpha$.
We state and prove some claims concerning a single element $H \in G$.
\begin{claim}\label{alpha} Let $H \in G$ and suppose that $\alpha$ is the first coordinate so that $\pi_\alpha(H)$ is non-degenerate.  Then $\alpha$ is not a limit ordinal.
\end{claim}
\begin{proof}(of Claim \ref{alpha})
If $\alpha$ is a limit ordinal then $M_\alpha$ is the inverse limit of the spaces $\{M_\gamma\}_{\gamma < \alpha}$.  Then if $\pi_\alpha(H)$ is non-degenerate, there must be a $\gamma < \alpha$ so that $f_\gamma^{\alpha}(\pi_\alpha(H))$ is non-degenerate which contradicts the hypothesis.
\end{proof}
Suppose that $\alpha$ is defined as in Claim 1; we are interested in the case where $\alpha>0$.  So suppose that $\pi_\alpha(H)$ is a subset of $f_{\alpha-1}^{-1}(p)$ for some point $p \in X_{\alpha - 1}$.  Then for $\gamma < \alpha$, $\pi_\gamma(H)$ is, by definition, a singleton.  For claims 2 and 3 suppose that $H \in G$ and $\alpha$ is defined as in claim 1.
\begin{claim}\label{pi} $\pi_\alpha(H)$ lies in a single element of $G_\alpha$.\end{claim}
\begin{proof}(of Claim \ref{pi}) If $\pi_\alpha(H)$ intersects two element of $G_\alpha$ then $\pi_{\alpha-1}(H)$ is non-degenerate; this contradicts the definition of $\alpha$.  \end{proof}
Suppose now that $x = \{x_\gamma\}_{\gamma < \omega_1} \in H$ and consider an arbitrary $\gamma$. Then for $\gamma < \alpha$ we have $x_\gamma = \pi_\gamma(H)$ since $\pi_\gamma(H)$ is degenerate in this case. For $\gamma = \alpha$, $\gamma$ is not a limit ordinal so $x_\gamma \in f_{\gamma-1}^{-1}(x_{\gamma -1})$; though it is possible that there is some point $y \in M$ with $y_\gamma \in f_{\gamma-1}^{-1}(x_{\gamma -1})$ so that $y \notin H$.  Suppose that $\gamma = \alpha+1$.  Then since $\alpha>0$, by construction $\pi_\alpha(H)$ is a pseudo-arc and there are two points $x_\alpha$ and $y_\alpha$ so that $\pi_\alpha(H)$ is irreducible from $x_\alpha$ to $y_\alpha$.  From the properties of continuous decompositions of the pseudo-arc, we also have the fact that $f_\alpha^{-1}(\pi_\alpha(H))$ is irreducible between each point of $f_\alpha^{-1}(x_\alpha)$ and each point $f_\alpha^{-1}(y_\alpha)$. So $\pi_\gamma(H) = \pi_{\alpha + 1}(H) = f_\alpha^{-1}(\pi_\alpha(H))$.  Therefore, since $\pi_{\alpha + 1}(H)$ is a union of elements of $G_{\alpha+1}$ for $\gamma > \alpha +1$ we have $\pi_\gamma(H) = {f_{\alpha+1}^\gamma}^{-1}(\pi_{\alpha+1}(H))$.  So this gives us:
\begin{claim}\label{H} $H=\{ \{x_\gamma\}_{\gamma < \omega_1} \ | \ x_\alpha \in \pi_\alpha(H)\}$.\end{claim}
For each non-degenerate subcontinuum $H$ of $M_{\omega_1}$ let $\alpha_H$ denote the first ordinal $\alpha$ such that $\pi_\alpha(H)$ is non-degenerate.  Then the following follows from the discussion preceding Claim \ref{H}:
\begin{claim} If $H\in G$ is a non-degenerate subcontinuum of $M_{\omega_1}$ then for each $\gamma > \alpha_H$ we have $H = \pi_\gamma^{-1} (\pi_\gamma(H))$.\end{claim}
\begin{claim}\label{K} Suppose $H \subset M_{\omega_1}$ is a non-degenerate continuum.  Then there is an open set $U$ in $C(M_{\omega_1})$ containing $H$ so that if $K \in U$ then $\alpha_K \le \alpha_H + 2$.\end{claim}
\begin{proof}(of Claim \ref{K})  Observe that the following collection is closed set in the hyperspace $C(M_{\omega_1})$ that misses $H$:
$$\{K \in C(M_{\omega_1}) | \alpha_K \ge \alpha_H + 2 \}.$$
\end{proof}
From Claim \ref{K} and the compactness of $M_{\omega_1}$ we have the following:
\begin{claim}\label{6}  There exists an ordinal $\lambda$ so that for each $ H\in G$ we have $\alpha_H < \lambda.$
\end{claim}
Consider now the collection $J=\{\pi_\lambda(H) | H \in G\}$ where $\lambda$ is the ordinal guaranteed by Claim \ref{6}.  Then from the above claims and the fact that all the relevant maps are open, $J$ is a continuous decomposition of $M_\lambda$. This gives us the  $M_{\omega_1}/G$ is homeomorphic to $M_\lambda/J$.  Since  $M_\lambda$ is a one dimensional metric continuum and $J$ is a continuous decomposition of $M_\lambda$ into pseudo-arcs, it follows that $M_\lambda/J$, and hence $M_{\omega_1}/G$, is a one dimensional metric continuum.
\end{proof}
\section{Generalized Whitney maps on $M_{\omega_1}$}
In \cite{Hernandez} Hern\'andez-Guti\'errez proved that there is a generalized Whitney map from the non-metric pseudo-arc of Smith onto the long arc.   In the following we generalize his result to the continua $M_{\omega_1}$ constructed in Section 3.

Suppose that $A$ is a Hausdorff arc with an order relation $\lhd$ generating the topology; suppose further that $a \lhd b$ are the non-cut points of $A$.  Suppose that $X$ is a continuum then the statement that $\mu: C(X) \to A$ is a generalized Whitney map means that $\mu$ is a continuous function and:
\begin{eqnarray*}
\mbox{If } x \in X  & \mbox{then} & \mu(\{x\}) = a, \\
\mbox{If } H \varsubsetneq K \in C(X)  & \mbox{then} & \mu(H) \lhd \mu(K).
\end{eqnarray*}

Let $\mathbb{L}= \omega_1 \times [0,1) \cup \{\omega_1\}$ with the following order topology, where we use the symbol $<$ for the usual orders on the sets $\omega_1$ and $[0,1)$:
\begin{eqnarray*}
\omega_1 & = & \mbox{min}(\mathbb{L}) \\
(\alpha, r) & \lhd & (\beta, s) \mbox{ if } \alpha > \beta \\
(\alpha, r) & \lhd & (\beta, s) \mbox{ if } \alpha = \beta \mbox{ and } r > s.
\end{eqnarray*}

We call $\mathbb{L}$ with the $\lhd$ order the inverted long arc.  For notational convenience we define $(\alpha, 1) = (\alpha+1, 0)$.

\begin{theorem}\label{Whitney} Let $M_{\omega_1}=\varprojlim\{M_\alpha, f_\alpha^\beta\}_{\alpha<\beta<\omega_1}$ be a non-metric one-dimensional continuum constructed in Section 3. Then there is a generalized Whitney map $\mu: C(M_{\omega_1}) \to \mathbb{L}$.
\end{theorem}
Since the one-dimensional continua $M_{\omega_1}$ can be decomposed into hereditarily indecomposable continua to yield a metric one-dimensional continua, the techniques of  Hern\'andez-Guti\'errez \cite{Hernandez} can be used to construct the Whitney map.  For the sake of completion, we provide an outline of a construction consistent with our approach.  We will need the following fact about continuous decompositions of one-dimensional metric continua as constructed above (see \cite{Le}). Suppose that $M$ is a one-dimensional metric continuum, $X$ is a one-dimensional metric continuum and $G$ is the continuous decomposition of $X$ into pseudo-arcs so that $M = X/G$.  Let $H \subset X$.  If $H$ intersects two elements of $G$, then $H = \cup \{g \in G | H \cap g \ne \emptyset \}$. In particular we have the following property:

\vspace{0.25cm}
\noindent
\textbf{Property D:} \textit{If $H$ is a subcontinuum of $X$, then $H$ is either contained in a single element of the decomposition, or it is a union of decomposition elements.}

\vspace{0.25cm}
\noindent
\textit{Notation:} Given the above, where $f: X \to M$ is the open monotone map associated with the decomposition space, if $H \subset M$ then $f(H) = \cup \{ f(x) | x \in H \}$.

\begin{lemma}  Suppose that $M$ is a one-dimensional metric continuum, $X$ is a one-dimensional metric continuum and $G$ is the continuous decomposition of $X$ into pseudo-arcs so that $M = X/G$.  Suppose that $\mu_M: C(M) \to [0,1]$ is a Whitney map with $\mu_M(M) = 1$.  Then there exists a Whitney map $ \mu_X: C(X) \to [0,2]$ so that $\mu_X(f^{-1}(H)) = \mu_M(H) + 1$ for all $H \in C(M)$.
\end{lemma}
\begin{proof}
Let $\nu: C(X) \to [0,2]$ be an onto Whitney map.  If $K$ is a subcontinuum of some element of $G$ then let $g_K$ denote that element; for $K \in C(X)$ define

$\mu_X(K) = \left \{
\begin{array}{cl}
\frac{\nu(K)}{\nu(g_K)} & \mbox{if } K \subset g_K \\
\mu_M(f(K)) + 1 & \mbox{if } K \nsubseteq g, \forall  g \in G
\end{array} \right .$

Since for each $g \in G$, $\nu(g) \ne 0$, the function $\frac{\nu(K)}{\nu(g_K)}$ is well defined and continuous.  Furthermore, if $K = g$ for some $g \in G$ we have
\begin{eqnarray*}
\frac{\nu(K)}{\nu(g_K)} & = & \frac{\nu(g_K)}{\nu(g_K)} = 1.
\end{eqnarray*}
From the definition of $\nu$ and property D, it follows if $H, K \in C(X)$ and $H \subsetneq K$ then $\mu_X(H) < \mu_X(K)$.
\end{proof}
Now we are ready to prove Theorem \ref{Whitney}.
\begin{proof}
Let $M_0$ be an arbitrary one dimensional metric continuum.  For $\alpha > 0$,
$G_\alpha = \{ f_{\alpha-1}^{-1})p) | \ p \in X_{\alpha-1} \}$ is the continuous decomposition of $M_\alpha$ into pseudo-arcs used to construct $M_\alpha$.
For each $\alpha$ let $\mu_\alpha: C(M_\alpha) \to [0,1]$ be a Whitney map.  For each successor ordinal there is a Whitney map $\widehat{\mu_\alpha} : M_\alpha \to [0,1]$ so that $\widehat{\mu_\alpha}(f_{\alpha -1}^{-1}(K)) = 1 + \mu_{\alpha}(K)$ for each $K \in C(M_{\alpha-1})$. We now construct $\mu: C(M_{\omega_1}) \to \mathbb{L}$ inductively.  For each $\alpha < \omega_1$ let $J_\alpha$ denote the set of subcontinua of $M_{\omega_1}$ so that $H \in J_\alpha$ if and only if $\alpha$ is the first ordinal such that $\pi_\alpha(H)$ is nondegenerate.  Then $J_\alpha \cap J_\beta = \emptyset$ if $\alpha \ne \beta$ and $C(M_{\omega_1}) = \cup_{\alpha < \omega_1}J_\alpha \cup \{ \{x\} | x \in M_{\omega_1} \}$. Let $\mu_0: C(M_0) \to [0,1]$ be a Whitney map.  Suppose $H \in J_0$ then define:
$$\mu(H) = (0, 1- \mu_0(H)).$$
Observe that if $H$ and $K$ are two elements of $J_0$ and $H \subsetneq K$ then $\mu(H) \lhd \mu(K)$; also observe that the continuity of $\mu_0$ implies that $\mu|_{J_0}$ is continuous.  Also observe that $\mu(M_{\omega_1})$ is the maximal element of $\mathbb{L}$ with respect to the order $\lhd$.
Suppose now that $\alpha$ is such that $\mu(H)$ has been constructed for all $H \in J_\beta$ with $\beta < \alpha$. 

\noindent
\textbf{Case 1: } $\alpha$ is a limit ordinal.  Observe that for limit ordinals $J_\alpha = \emptyset$.

\noindent
\textbf{Case 2:} $\alpha$ is a successor ordinal and $H \in J_\alpha$.  Then $\mu(K)$ had been defined for $K \in \cup_{\beta<\alpha}J_\beta$ so that $\mu(K) \in [(\alpha, 1), (0,0)]$ and restricted to $\cup_{\beta<\alpha}J_\beta$, $\mu$ has the Whitney property.  Since $M_{\alpha-1}$ is a continuous decomposition of $M_\alpha$ into pseudo-arcs satisfying property D, we wish to apply the lemma.  Since topologically $[ (\alpha, 1),(0,0)]_\lhd $ is equivalent to $[0,1]$ and $[(\alpha -1, 1),(\alpha-1, 0)]_\lhd $ is topologically $[1,2]$ we can let $\phi: [0,2] \to [(\alpha -1, 1),(\alpha, 0)]$ be a homeomorphism so that $\phi(1) = (\alpha, 1)$.  Then we apply the lemma to obtain $\mu(H)$ in terms of $\mu_{M_\alpha}(H)$. 

Finally for $x \in M_{\omega_1}$ define $\mu(\{x\}) = \omega_1$. If $H \in J_\alpha$, $K \in J_\beta$ and $\alpha < \beta$ then $H \nsubseteq K$ and $\mu(K) \lhd \mu(H)$.  If $H, K \in J_\alpha$ and $H \subset K$ then by construction $\mu(H) \lhd \mu(K)$.  So the fact that $\mu$ has the Whitney property can be easily verified from the fact that at each stage of the construction the Whitney property is preserved.

We wish to verify the continuity of $\mu$: First observe that for $H \in C(M_{\omega_1})$, if $\alpha$ is the first ordinal so that $\pi_\alpha(H)$ is not degenerate then $\alpha$ is not a limit ordinal.  So $ (\alpha,1) \lhd  \mu(H) \trianglelefteq (\alpha, 0)= (\alpha-1,1)$; then for some $t <1$  there is an open set $U$ in $C(M_{\omega_1})$ so that if $K \in U$ then $ (\alpha, 0) \lhd  \mu(K) \lhd (\alpha-1, t)$.  So $\mu$ is continuous at $H$.   Consider $x \in M_{\omega_1}$ and let $\alpha < \omega_1$, observe that $\cup_{\beta > \alpha} J_\beta \cup \{ \{ x \} | x \in M_{\omega_1} \}$ is a neighborhood of $\{x\}$.  So, from the construction of $\mu$ this gives us continuity at $\{x\} \in  C(M_{\omega_1})$.
\end{proof}
\section{Questions}
We conclude with the following natural questions, motivated by the results for the metric pseudo-arc and pseudo-circle.
\begin{question}
Is Smith's nonmetric pseudo-arc a unique homogeneous hereditarily indecomposable arc-like nonmetric Hausdorff continuum; cf. \cite{Bi4}?
\end{question}
\begin{question}
Does Smith's nonmetric pseudo-arc has the near-homeomorphism property; cf. \cite{WL}, \cite{MS2}?
\end{question}
\begin{question}
Are Cartesian products of Smith's nonmetric pseudo-arcs factor-wise rigid; cf. \cite{BellamyLysko}, \cite{BellamyKennedy}?
\end{question}
\begin{question}
Does the nonmetric pseudo-circle from Corollary \ref{pseudocircle} admit a minimal homeomorphism; cf. \cite{Ha}?
\end{question}
\section{Acknowledgments}
The authors are grateful to Rodrigo Hern\'andez-Guti\'errez for helpful comments that improved the paper. J. Boro\'nski's work was supported by National Science Center, Poland (NCN), grant no. 2015/19/D/ST1/01184 for the project \textit{Homogeneity and Minimality in Compact Spaces}.

\bibliographystyle{amsplain}

\begin{thebibliography}{40}
\bibitem{BellamyLysko} {\sc Bellamy, D. P.; Lysko, J. M.} {\em Factorwise rigidity of the product of two pseudo-arcs.} \textbf{Topology Proc.} 8 (1983), pp. 21–27.
\bibitem{BellamyKennedy} {\sc Bellamy, D. P.; Kennedy, J. A.} {\em Factorwise rigidity of products of pseudo-arcs.} \textbf{Topology Appl.} 24 (1986), pp. 197--205. 
\bibitem{Bi2} {\sc Bing, R.H. } {\em A homogeneous indecomposable plane continuum,} \textbf{Duke Math.} J. 15 (1948), pp. 729--742.
\bibitem{Bi} {\sc Bing, R.H. } {\em Concerning hereditarily indecomposable continua,} \textbf{Pacific J. Math.} 1 (1951), pp. 43--51.
\bibitem{Bi4}{\sc Bing, R. H.} {\em Each homogeneous nondegenerate chainable continuum is a pseudo-arc.} \textbf{Proc. Amer. Math. Soc.} 10 (1959) pp. 345--346.
\bibitem{Bi3} {\sc Bing, R.H. and F. B. Jones}, {\em Another homogeneous plane continuum,} \textbf{Trans. Amer. Math. Soc.},
90 (1959), pp. 171--192.
\bibitem{Bo}{\sc Boro\'nski, J. P.} {\em On the number of orbits of the homeomorphism group of solenoidal spaces.} \textbf{Topology Appl.}, 182, (2015), pp. 98--106
\bibitem{BoronskiSmith}{\sc Boro\'nski, J. P.; M. Smith} {\em On the conjecture of Wood and projective homogeneity} 	arXiv:1607.04105 [math.GN].
\bibitem{En} {Engelking, R.} {\em Zarys topologii ogólnej. (Polish) [Outline of general topology]} Biblioteka Matematyczna, Tom 25 Państwowe Wydawnictwo Naukowe, Warsaw, 1965.
\bibitem{Fe}{\sc Fearnley, L.} {\em The pseudo-circle is not homogeneous,} \textbf{Bull. Amer. Math. Soc.} 75 (1969), pp. 554--558.
\bibitem{Ha} {\sc Handel, M.}, {\em A pathological area preserving $C^\infty$ diffeomorphism of the plane.} \textbf{Proc. Amer. Math. Soc.}, 86  (1982), pp. 163--168.
\bibitem{Hernandez}{\em Hern\'andez-Guti\'errez, R.} {\em A Whitney map onto the long arc.} \textbf{Questions Answers Gen. Topology} 32 (2014), pp. 145--152.
\bibitem{HoOv} {\sc Hoehn, L. C.; Oversteegen, L. G.} {\em A complete classification of homogeneous plane continua.} \textbf{Acta Math.} 216 (2016), pp. 177--216.
\bibitem{RiPP}{\sc Jim\'enez-Hern\'andez R., Minc P., Pellicer-Covarrubias P.} {\em A family of circle-like, $\frac{1}{n}$-homogeneous, indecomposable continua.} \textbf{Topology Appl.}, 160 (2013), pp. 930--936.
\bibitem{Solecki}{\sc Irwin, T.; Solecki, S.} {\em Projective Fra\"iss\'e limits and the pseudo-arc.} \textbf{Trans. Am. Math. Soc.} 358, pp. 3077--3096 (2006).
\bibitem{Kawamura}{\sc Kawamura, K.}, {\em On a conjecture of Wood}, \textbf{Glasg. Math. J.} 47 (2005), pp. 1--5.
\bibitem{KeRo} {\sc Kennedy, J. and Rogers, J. T., Jr.} {\em Orbits of the pseudocircle,} \textbf{Trans. Amer. Math. Soc.} 296 (1986), pp. 327--340.
\bibitem{Knaster} {\sc Knaster, B.} {\em Un continu dont tout sous-continu est indecomposable,} \textbf{Fund. Math.} 3 (1922), pp. 247--286.
\bibitem{Krasinkiewicz}{\sc Krasinkiewicz, J.} {\em On homeomorphisms of the Sierpi\'nski curve.}
\textbf{Prace Mat.} 12 (1969), pp. 255--257.
\bibitem{Kwiatkowska} {\sc Kwiatkowska, A.} {\em Large conjugacy classes, projective Fra\"iss\'e limits and the pseudo-arc.} \textbf{Israel J. Math. }201 (2014), pp. 75--97.
\bibitem{WL}{\sc Lewis, W.}, {\em Most maps of the pseudo-arc are homeomorphisms}, Proc. Amer. Math. Soc.  91  (1984),  no. 1, 147--154.
\bibitem{Le} {\sc Lewis, W.}, {\em Continuous curves of pseudo-arcs.} \textbf{Houston J. Math. }11 (1985), pp. 91--99.
\bibitem{MisloveRogers} {\sc Mislove, M.; Rogers, James T., Jr.} {\em Local product structures on homogeneous continua.} \textbf{Topology Appl.} 31 (1989), pp. 259–267.
\bibitem{Mo} {\sc Moise, E. E.} {\em An indecomposable plane continuum which is homeomorphic to each of its nondegenerate subcontinua,}
\textbf{Trans. Amer. Math. Soc.} 63, (1948), pp. 581--594.
\bibitem{Op} {\sc Oppenheim, I.}, MSc. thesis, Tel Aviv University, 2008.
\bibitem{Rambla} {\sc Rambla, F.} {\em A counterexample to Wood's conjecture}, \textbf{J. Math. Anal. Appl.} 317 (2006), pp. 659--667.
\bibitem{Ro2} {\sc Rogers, J.T., Jr.} {\em The pseudo-circle is not homogeneous,} \textbf{Trans. Amer. Math. Soc. } 148 (1970), pp. 417--428.
\bibitem{MS2} {\sc Smith, M.}, {\em Every mapping of the pseudo-arc onto itself is a near homeomorphism.}
\textbf{Proc. Amer. Math. Soc.} 91 (1984), 163--166.
\bibitem{MS} {\sc Smith, M.}, {\em On nonmetric pseudo-arcs} \textbf{Top. Proc.} 10 (1985), pp. 385-397.
\bibitem{Sp} {\sc Spanier, E. H.} {\em Algebraic topology}, first corrected Springer ed., Springer-Verlag, New York, 1966.
\end{thebibliography}

\end{document}